\documentclass[11pt]{amsproc}
\usepackage{amsfonts,amsmath,amsthm}
\usepackage{amssymb}
\usepackage{url}
\usepackage{hyperref}
\usepackage[utf8]{inputenc}
\usepackage[english]{babel}
\usepackage{graphicx}

\newtheorem{theorem}{Theorem}[section]
\newtheorem{lemma}[theorem]{Lemma}

\newtheorem{proposition}[theorem]{Proposition}

\newtheorem{definition}[theorem]{Definition}

\theoremstyle{remark}

\newtheorem{example}[theorem]{Example}
\newtheorem{remark}[theorem]{Remark}

\def\K{\mathbb K}
\def\R{\mathbb R}
\def\C{\mathbb C}

\def\Epi{\operatorname{Epi}}

\def\norma{\|\cdot\|}

\title{A natural correspondence between quasiconcave functions and fuzzy norms}

\begin{document}

\author
{Javier~Cabello~Sánchez, Daniel~Morales~González}

\date{}

\thanks
{Departamento de Matem\'{a}ticas and Instituto de Matemáticas. 
Universidad de Extremadura,
Avda. de Elvas s/n, 06006, Badajoz, Spain \\
email:\ coco@unex.es (corresponding author), danmorg@unex.es}  



\begin{abstract}
In this note we show that the usual notion of fuzzy norm defined on a linear 
space is equivalent to that of quasiconcave function, in the sense that every 
fuzzy norm $N:X\times\mathbb{R}\to[0,1]$ defined on a (real or complex) linear space $X$ 
is uniquely determined by a quasiconcave function $f:X\to[0,1]$. We explore the 
minimum requirements that we need to impose to some quasiconcave function 
$f:X\to[0,1]$ in order to define a fuzzy norm $N:X\times\mathbb{R}\to[0,1]$. 

Later we use this equivalence to prove some properties of fuzzy norms, like 
a generalisation of the celebrated Decomposition Theorem. 
\end{abstract}

\maketitle

\section{Introduction}
Since the seminal paper \cite{bag2003finite}, where T.~Bag and S.K.~Samanta 
introduced their definition of fuzzy normed linear space and proved that, under 
some circumstances, a fuzzy norm can be seen as a collection of {\em usual} 
norms, a lot of work has been carried out in fuzzy analysis. Some results are 
essentially adaptations of classic results in functional analysis; see, e.g.,
\cite{Font2017, HBF1999}, while other deal with purely fuzzy 
concepts; see~\cite{MatrixNorm, Deco2020, QFN}. 

The goal in this note is to show that any fuzzy norm, in the sense of 
\cite{bag2003finite}, $N:X\times\mathbb{R}\to[0,1]$ defined on a 
(real or complex) linear space $X$ is uniquely determined by a 
quasiconcave function $f:X\to[0,1]$. It should be noted that the 
relation between both concepts is not difficult to prove and that in 
some sense it was noticed in~\cite{Ramik}, but there has been no further 
development, maybe due to the fact that~\cite{Ramik} appeared one year 
before~\cite{bag2003finite}. 
As this kind of functions (and specially their counterpart, the quasiconvex 
functions) has been thoroughly studied throughout the years, this equivalence 
may give rise to a great opportunity to begin taking advantage of tools developed 
in nonconvex optimization and see what can be obtained in fuzzy analysis 
-- see~\cite{Handbook} and, more precisely,~\cite{Crou} and the references therein. 

Later, in Section~\ref{Consequences}, we show that some results about the 
continuity of quasiconcave functions lead to the equivalence between 
fuzzy convergence of sequences and usual convergence of sequences in 
finite-dimensional spaces. We finish this paper proving a more general 
statement of the Decomposition Theorem that the reader may find 
in~\cite{bag2003finite}. 

\begin{definition}
Let $X$ be a non-empty set. A fuzzy set in $X$ is a function 
$\mu:X\to[0,1]$. 
\end{definition}

\begin{definition}[\cite{bag2003finite}]\label{defuzzy}
Let $X$ be a vector space over the field $\K=\R,\C$. A {\em fuzzy norm} 
on $X$ is a fuzzy set $N$ in $X\times[0,\infty)$ satisfying:
\begin{enumerate}
\item[(N1)] $N(x,0)=0$, for all $x\in X$;
\item[(N2)] $[N(x,t)=1$, for all $t>0]$ if and only if $x=0$;
\item[(N3)] $N(\lambda x,t)=N\left(x,\frac{t}{|\lambda|}\right)$, 
for all $t\geq 0$, $x\in X$ and $0\neq\lambda\in\K$;
\item[(N4)] $N(x+y,t+s)\geq \min\{N(x,t), N(y,s)\}$, for all $x,y\in X$ and $t,s\geq 0$;
\item[(N5)] For each $x\in X$, $\lim_{t\to\infty}N(x,t)=1$.
\end{enumerate}
We will say that the pair $(X,N)$ is a {\em fuzzy normed linear space} (briefly, {\em FNLS}).
\end{definition}

\begin{remark}
When defining the concept of fuzzy norm, instead of defining $N:X\times\R\to[0,1]$ 
and imposing that $N(x,t)=0$ for every $t<0$, we have chosen $N:X\times[0,\infty)\to[0,1]$. 
Of course, both definitions are equivalent. 
\end{remark}

Observe that for each $x\in X$, the function $N_x:[0,\infty)\rightarrow [0,1]$, 
$N_x(t)=N(x,t)$ is non-decreasing. Indeed, consider $t<s$, applying (N2) and (N4) we have
$$N_x(s)=N(x,s)\geq\min\{N(0,s-t), N(x,t)\}=N(x,t)=N_x(t).$$

\begin{example}[\cite{Font2017}]
Given a normed space $(X,\norma)$, the function 
$$N_{\norma}:X\times[0,\infty)\rightarrow [0,1]$$
defined by
$$N_{\norma}(x,t)=\dfrac{t}{t+\|x\|}, \qquad N_{\norma}(0,0)=0,$$
is a fuzzy norm on $X$ and is called the  {\em standard fuzzy norm.}
\end{example}

\begin{definition}
Given a (complex or real) linear space $X$, a function $f:X\rightarrow\R$ is 
said to be {\em quasiconcave} if for all $x,y\in X$ and $0<\lambda<1$, the 
following holds
$$ f(\lambda x+(1-\lambda)y)\geq\min\{f(x), f(y)\}. $$
\end{definition}

\section{The main result}

Throughout this section, $X$ will denote a vector space over the field 
$\K=\R,\C$, and we will denote by $\mathcal{F}$ the set of fuzzy norms on $X$.

Let $\mathcal{A}$ be the set of quasiconcave functions $f:X\rightarrow\R$ such that 
\begin{itemize}
\item[(A1)] $f(0)=1$ and, if $f(tx)=1$ for every $t$, then $x=0$;
\item[(A2)] $\lim_{t\to 0} f(tx)=1$ for every $x$;
\item[(A3)] $f(\lambda x)=f(|\lambda|x)$ for all $\lambda\in\K, x\in X$.
\end{itemize}

\begin{theorem}[Characterisation of fuzzy norms]\label{QCF}
The map 
$$\begin{array}{ccl}
\mathcal{A} & \longrightarrow & \mathcal{F}\\
f & \longmapsto & N_f(x,t)=
\begin{cases}
f\left(\frac{x}{t}\right) &\text{if}\quad t\neq 0;\\
0 &\text{if}\quad t=0;\\
\end{cases}
\end{array}$$
is bijective.
\end{theorem}

\begin{proof}
We will show that given $f\in\mathcal{A}$, $N_f$ is a fuzzy norm on $X$ 
--it is clear that if the map is well defined, then it is injective:
\begin{enumerate}
\item[(N1)] $N_f(x,0)=0$, for all $x\in X$ by the definition of $N_f$. 
\item[(N2)] $N_f(0,t)=f\left(\frac{0}{t}\right)=1$, for all $t>0$.
If $x\neq 0$, then (A1) implies that there exists $t\in\R$ such that 
$N_f(x,t)=f\left(\frac{x}{t}\right)\neq 1$. 
\item[(N3)] $N_f(\lambda x,t)=f\left(\frac{\lambda x}{t}\right)=f\left(\frac{|\lambda| x}{t}\right)=
N_f\left(x,\frac{t}{|\lambda|}\right)$, for all $t\geq 0$, $x\in X$ and $0\neq\lambda\in\K$. 
\item[(N4)] For every $x,y\in X$ and $t,s\geq 0$
\begin{eqnarray*}
N_f(x+y,t+s)&=& f\left(\frac{x+y}{t+s}\right)\\
&=& f\left(\frac{t}{t+s}\cdot\frac{x}{t}+\frac{s}{t+s}\cdot\frac{y}{s}\right)\\
&\stackrel{(*)}\geq & \min\left\lbrace f\left(\frac{x}{t}\right),f\left(\frac{y}{s}\right)\right\rbrace\\
&=& \min\{N_f(x,t), N_f(y,s)\};
\end{eqnarray*}
where the inequality $\stackrel{(*)}\geq$ holds because $f$ is quasiconcave. 
\item[(N5)] Given $x\in X$, $\lim_{t\to\infty}N_f(x,t)=\lim_{t\to\infty}
f\left(\frac{x}{t}\right)=\lim_{t\to 0}f(tx)=1$. 
\end{enumerate}

\medskip

Now we need to show that, given $N\in\mathcal{F}$, the function $f_N(x)=N(x,1)$ 
belongs to $\mathcal{A}$ --with this we have that the map is a surjection.
\begin{itemize}
\item[(A0)] $f_N$ is quasiconcave. Indeed, for every $x,y\in X,\lambda\in[0,1]$, 
one has:
\begin{eqnarray*}
f_N(\lambda x+(1-\lambda)y)&=& N(\lambda x+(1-\lambda)y,1)\\
&\geq & \min\{N(\lambda x,\lambda),N((1-\lambda)y,1-\lambda)\}\\
&=& \min\{N(x,1),N(y,1)\}\\
&=& \min\{f_N(x),f_N(y)\}. 
\end{eqnarray*}
\item[(A1)] $f_N(0)=N(0,1)=1$ and, if $x\neq 0,$ then there exists $t$ such that 
$f_N(x)=N(x,1)\neq 1$ by (N2). 
\item[(A2)] $\lim_{t\to 0} f_N(tx)=\lim_{t\to 0} N(tx,1)=\lim_{t\to \infty} N(x,t)=1$ by (N5). 
\item[(A3)] $f_N(\lambda x)=N(\lambda x,1)=N(|\lambda| x,1)=f_N(|\lambda| x)$ by (N3).
\end{itemize}
So, both concepts are equivalent. 
\end{proof}

\section{Applications}\label{Consequences}
In this section we analyse some consequences of the equivalence of quasiconcave
functions and fuzzy norms. First we show that the only topology defined by 
fuzzy norms in finite-dimensional spaces is the topology given by any norm. Later we give a generalisation of the Decomposition Theorem. 

\subsection{Finite-dimensional spaces}
In this subsection, we will show some results that improve the ones that 
can be found in~\cite[Section 3]{sadeqi} for the case of finite-dimensional 
spaces. Namely, while in ~\cite[Section 3]{sadeqi} we can find results on 
infinite-dimensional spaces, all of them depend on the Conditions (N6) and (N7). 

\begin{itemize}
\item[(N6)] $[\forall t>0, N(x, t)>0]$ implies $x=0$.
\item[(N7)] For any non-zero element $x$, $N(x,\,.\,)$ is a continuous function on $\R$ and 
strictly increasing on $\{t:0<N(x,t)<1\}$. 
\end{itemize}

We will use the following results, stated as Proposition 3.8 and Theorem 3.2 
in~\cite{Crou} for quasiconvex functions. Please take into account that a 
function $f$ is quasiconvex if and only if $-f$ is quasiconcave, so 
in~\cite{Crou}, the inequality that appears in the next Proposition is reversed. 

\begin{proposition}[\cite{Crou}]\label{continuaa}
Assume that $f:E\to \R$ is quasiconcave, $a,b\in E$, $f(b)>f(a)$ and $f$ is 
continuous at $b$. Then $f$ is continuous at $a$ whenever 
$t\mapsto f(a+t(a-b))$ is continuous at $t=0$. 
\end{proposition}

\begin{theorem}[\cite{Crou}]\label{continua0}
Assume that $f:\R^n\to\R$ is quasiconcave. Then, $f$ is continuous at $a$ 
if and only if $t\mapsto f(a+tb)$ is continuous at $t=0$ for every $b\in\R^n$. 
\end{theorem}

\begin{definition}[\cite{bag2005bounded}, p. 536]
For $\K=\R$ or $\C$ we will denote as $N_\K$ the fuzzy norm in $\K$ defined as 
\begin{equation}\label{absval}
N_\K(x,t)=
\left\{
\begin{array}{c}
0 \text{ if } t\leq |x| \\
1 \text{ if } t>|x|
\end{array}
\qquad
\right.
\end{equation}
\end{definition}

\begin{definition}[\cite{bag2005bounded}, p. 524]
A mapping $T$ from $(U, N_1)$ to $(V, N_2)$ is said to be fuzzy continuous at 
$x_0\in U$ if for every $t>0, \alpha\in(0,1)$ there exist $s>0, \beta\in(0,1)$ 
such that 
$$N_1(x-x_0,s)>\beta \Rightarrow N_2(T(x)-T(x_0),t)>\alpha.$$
\end{definition}

\begin{definition}[\cite{bag2005bounded}, p. 524]
A mapping $T$ from $(U, N_1)$ to $(V, N_2)$ is said to be sequentially fuzzy 
continuous at $x_0\in U$ if 
\begin{equation}\label{seqcon}
\lim_{n\to\infty}N_1(x_n-x_0,t)=1,\forall\,t>0\Rightarrow 
\lim_{n\to\infty}N_2(T(x_n)-T(x_0),t)=1,\forall\,t>0. 
\end{equation}
\end{definition}

\begin{theorem}[\cite{bag2005bounded}, Theorem~3.2]\label{thmfscc}
Let $T:(U,N_1)\to (V,N_2)$ be a mapping where $(U,N_1)$ and $(V,N_2)$ are FNLS. 
Then T is fuzzy continuous iff it is sequentially fuzzy continuous.
\end{theorem}

With these definitions and results in mind, we can obtain the following. 

\begin{lemma}
Let $(X,N)$ be a finite-dimensional FNLS. Then, 
$N(\,\cdot\,,t_0):X\to\R$ is continuous at $x=0$ for every $t_0$.
\end{lemma}

\begin{proof}
Consider the quasiconcave  function $f:X\to[0,1]$ associated to $N$.
Conditions (A2) and (A1) give
$\displaystyle\lim_{t\to 0} N(tx,1)= \lim_{t\to 0} f(tx)=1=f(0)$
for every $x$, so Theorem~\ref{continua0} implies that $f$ is continuous at 0.
Equivalently, that $N(\,\cdot\,,1)$ is. For other positive $t_0$ we also have
$$\lim_{x\to 0}N(x,t_0)=\lim_{x\to 0}f(x/t_0)=1,$$
so $N(\,\cdot\,,t_0)$ is continuous at 0 for every $t_0>0$. When $t_0=0$ the
statement is trivial, so it is continuous for every $t_0\geq 0.$
\end{proof}

\begin{proposition}\label{sequence}
Let $(X,N)$ be a finite-dimensional FNLS and consider a sequence 
$(x_n)_n\subset X$. One has 
$$\lim_{n\to \infty}N(x_n,t)=1,\,\forall\, t>0\ \Longleftrightarrow\  
\lim_{n\to \infty}x_n=0.$$
\end{proposition}

\begin{proof}
If we have $(x_n)_n\subset X$ such that 
$\lim_{n\to\infty}N(x_n,t)=1,\,\forall\,t>0,$ then $(x_n)_n\to 0$ with respect
to the only linear topology with which we can endow $X$. Indeed, suppose on the
contrary that $(x_n)_n$ does not converge to 0. Substituting $t$ by 1,
we have $(f(x_n))_n\to 1$. Now we have two options; either $(x_n)_n$ is bounded,
in which case it has an accumulation point $y\in X$, or it is unbounded.

In the first case, passing if necessary to a subsequence, we may suppose that
$(x_n)_n\to y$. As $f$ is quasiconcave, for every $\lambda\in[0,1]$ one has
$$1= f(0)\geq\lim_{n\to\infty}f(\lambda x_n)\geq \lim_{n\to\infty}f(x_n)=1,$$
so for every $\lambda y$ with $0\leq\lambda\leq 1$ there is some sequence
that converges to $\lambda y$ and whose image through $f$ converges to 1.
As $t\in[0,\infty)\mapsto f(ty)\in[0,1]$ is a nonincreasing function, 
it is continuous almost everywhere, in particular it is continuous at some 
$\lambda\in(1/2,1)$. Then, $f(\lambda y)<1=f(0)$ and Proposition~\ref{continuaa} 
imply that $f$ is continuous at $\lambda y$. This means that 
$f(\lambda y)=\lim_{n\to\infty}f(\lambda x_n)=1$ so, in any 
case, $f(\lambda y)=1$. This implies that $f(\mu y)=1$ 
whenever $0\leq\mu\leq 1/2<\lambda$. In terms of $N$, this means that 
$1=N(\mu y,1)=N(y,1/\mu)=N(y,t)$ for every $t=1/\mu\geq 2$. Choose $s>0$ 
such that $N(y,s)<1$, such $s$ exists because of (N2). It is clear that 
$\lim_{n\to \infty}N(x_n,t)=1,\,\forall\, t>0$ is equivalent to 
$\lim_{n\to \infty}N(2x_n/s,2t/s)=1,\,\forall\, t>0$ and to 
$\lim_{n\to \infty}N(2x_n/s,t)=1,\,\forall\, t>0$. Now, the previous argument 
shows that $N(y,s)=N(2y/(2s),1)=1,$ a contradiction with the choice of $s$. 
This way, we obtain that $(x_n)_n$ must converge to 0.

As for the second case, suppose that $(x_n)_n$ has no accumulation point. 
We only need to choose a bounded neighbourhood $B$ of 0
and take into account that, for every $n$ such that $x_n\not\in B$, there is
$0<\lambda_n<1$ such that $\lambda_n x_n\in B\setminus \left(\frac 12B\right)$.
As $f$ is quasiconcave, $f(\lambda_nx_n)\geq\min\{f(0),f(x_n)\}=f(x_n)$
so we can substitute each $x_n$ by $\lambda_nx_n$ and apply the previous case.

So, whenever the fuzzy limit of a sequence is 0, the sequence converges and
its limit is 0.

To show that the other implication also holds we just need to observe that 
$\lim(x_n)_n=0$ implies $\lim (f(x_n))_n=1$ because $f(0)=1$ and, by the 
previous Lemma, $f$ is continuous at $0$. This gives $\lim N(x_n,1)=1$ and for 
any $t_0>0$ we can apply the same argument to $(f(x_n/t_0))_n$. 
\end{proof}

\begin{theorem}\label{thmcontinuous}
Let $(X,N)$ be a finite-dimensional FNLS and $t_0\geq 0$. Then, the fuzzy 
continuity of $N(\,\cdot\,,t_0)$ is equivalent to its continuity. In 
particular, it is always continuous at 0 and, if $N(x_0,t_0)<1$, then it 
is continuous at $x_0\neq 0$ if and only if $t\mapsto N(t,\,\cdot\,)$ 
is continuous at $t_0$. 
\end{theorem}

\begin{proof}
By Theorem~\ref{thmfscc} we know that $N(\,\cdot\,,t_0)$ is fuzzy continuous 
if and only if it is sequentially fuzzy continuous. It is a well-known fact that 
this also happens with continuity and sequential continuity in metric spaces, 
so it suffices to show that sequential continuity and fuzzy sequential 
continuity are equivalent, but this is immediate from Proposition~\ref{sequence}. 

Now, applying Proposition~\ref{continuaa} we obtain that $f(x)=N(x,1)$ is fuzzy 
continuous at $x_0$ if and only if $t\mapsto f(tx_0)$ is continuous at $t=1$. 
The general case for an arbitrary $t_0\in(0,\infty)$ is analogous, so we are done. 
\end{proof}

\begin{remark}
Observe that if we consider $(\R,N_\R)$ and the associated quasiconcave function 
$$f_\R(x)=N_\R(x,1)=
\left\{
\begin{array}{c}
0 \text{ when }1\leq |x|\\
1 \text{ when }1>|x|
\end{array}
\right.$$
$f_\R$ is not (fuzzy) continuous at $\pm 1$ and neither is 
$N_\R(\,\cdot\,,t_0)$ at $\pm t_0$ with $t_0>0$. 

On the other hand, if Condition (N7) holds then $N$ is continuous 
$N(x_0,t_0)$ whenever $N(x_0,t_0)<1$, provided the space is finite-dimensional. 
\end{remark}

\subsection{Decomposition Theorem}
After a thorough reading of~\cite{bag2005bounded} and \cite{nadaban}, we have 
realized that the Condition (N6) in the Decomposition Theorem that appears as 
Theorem~2.2~in~\cite{bag2003finite} is unnecessarily restrictive. This can be 
seen as a consequence of the intuition that quasiconcave functions give rather 
than a consequence of any results on this kind of functions. Namely, it is 
well-known that a real function $f:X\to\R$ is quasiconcave  
if and only if every $S(f,\lambda)=\{x\in X:f(x)\geq \lambda\}$ 
is convex. As the function $f_N$ always fulfils $f_N(x)=f_N(-x)$, these 
epigraphs are also symmetric. The only that each $\Epi(f,\lambda)$ needs to fulfil 
in order to define a norm via its Minkowsky functional is that it is bounded 
and absorbing and these properties are easily characterized by means of the 
behaviour of $f(tx)$ when $t\to 0$ and $t\to\infty$. Summarizing, we have 
the following result, that generalises Theorem~2.2~in~\cite{bag2003finite}. 

\begin{theorem}[Decomposition Theorem,~\cite{bag2003finite}]\label{WeDeco}
Let $(X, N)$ be an FNLS. Assume further that,
\begin{itemize}
\item[(N6$'$)] For every $x\in X$, $N(x,t)$ converges to $0$ when $t\to 0$ 
\end{itemize}
and define for each $0<\alpha<1$ the function
\begin{equation*}
p_\alpha(x)=\inf\{t>0:N(x,t)>\alpha\}=\inf\{t>0:f_N(x/t)>\alpha\}. 
\end{equation*}
Then $\mathcal{P}=\{p_\alpha:\alpha\in(0,1)\}$ is an ascending family of norms 
on $X$. 
\end{theorem}

\begin{proof}
For the first part of the proof we need to show that each $p_\alpha$ fulfils 
the three conditions of the definition of norm. Namely, 

(i)\quad  
It vanishes exactly at $x=0.$ 
Since $N(0,t)=1$ for every $t>0$, we have that $p_\alpha(0)=0$. 
Conversely, if $p_\alpha(x)=0$ then
$x=0$ because for every $x\neq 0$, $\lim_{t\to 0} N(x,t)=0$ by (N6$'$).

(ii)\quad  
It is positively homogeneous: 
\begin{eqnarray*}
p_\alpha(\lambda x)&=& \inf\{t>0:N(\lambda x,t)>\alpha\}\\
&=& \inf\left\lbrace t>0:N\left(x,\dfrac{t}{|\lambda|}\right)>\alpha\right\rbrace\\
&=& \inf\{|\lambda| t>0:N(x,t)>\alpha\}\\
&=& |\lambda| \inf\{ t>0:N(x,t)>\alpha\}\\
&=& |\lambda|p_\alpha(x).
\end{eqnarray*}

(iii)\quad  
And fulfils the triangle inequality: \\ 
If $f(x/t)>\alpha$ and $f(y/s)>\alpha$, then the definition of quasiconcave function 
gives, for every $\lambda\in[0,1],$ the inequality 
$f(\lambda x/t+(1-\lambda)y/s)>\alpha$. With $\lambda=t/(t+s)$ we obtain
$$\alpha<f\left(\frac{t}{t+s} x/t+\frac{s}{t+s}y/s\right)=f((x+y)/(t+s)).$$
So, $p_\alpha(x)\leq t$ and $p_\alpha(y)\leq s$ imply $p_\alpha(x+y)\leq t+s.$

To show that $\mathcal{P}=\{p_\alpha\}_{0<\alpha<1}$ is an ascending family, 
we just need to observe that, for $\alpha\geq\beta$ one has 
\begin{eqnarray*}
        p_\alpha(x)&  =  & \inf\{t>0:N(x,t)>\alpha\}\\
         &\leq & \inf\{t>0:N(x,t)>\beta\}=p_\beta(x),
\end{eqnarray*}
because $N(x,t)>\beta$ implies $N(x,t)>\alpha$.
\end{proof}

\begin{remark}
The original statement of Theorem~\ref{WeDeco} includes as a hypothesis 
that $N(x,\,\cdot\,)$ must vanish eventually, i.e., that for every $x$ 
there exists $t_x>0$ such that $N(x, t)=0 $ for every $t\leq t_x$. It is written 
in a different way, 
but it is clear that our Condition (N6$'$) is much weaker 
than (N6) and it does not hinder the proof. 
\end{remark}

\subsection{Concluding comments and remarks}
We have seen that the easy observation labelled as Theorem~\ref{QCF} is 
useful when it comes to analyse the continuity of fuzzy norms, and it will 
probably be useful in other aspects of fuzzy functional analysis. 

We think that it is worth pointing out that there is something that does 
not feel right about the choice of $(\K,N_\K)$ as the basic structure amongst 
the fuzzy normed linear spaces, as in the definition of the dual FNLS that 
the reader can find in~\cite[Definition~5.1]{bag2005bounded}. 
Of course, $(\K,N_\K)$ can be seen as the most natural structure 
of FNLS. Indeed, if for any fuzzy norm $N$ we think $N(x,t)$ as the truth 
value of the statement ``the norm of $x$ is less than or equal to the real 
number $t$'' (see, e.g., \cite{MMM}), then $N_\K(x,t)$ equals 1 if $|x|\leq t$ 
and vanishes if $|x|>t$, so the FNLS structure is directly inherited from 
the normed space $(\K,|\,\cdot\,|)$. But, as the fuzzy norm endows the linear 
space with a topological structure, one could expect the fuzzy norm to be 
continuous. The problem that we see with the standard fuzzy norm 
defined as $N_{|\,\cdot\,|}(x,t)=\displaystyle\frac t{t+|x|}$ for every 
$x\in X, t>0$ is that it does not fulfil the Condition (N6), but thanks to 
Theorem~\ref{WeDeco} this Condition is no longer needed in order to get a 
decomposition of $N$ as a family of crisp norms and the fuzzy norm 
$N_{|\,\cdot\,|}$, unlike $N_\K$, fulfils Condition (N7). 

\section*{Acknowledgements}
This work has been partially supported by DGICYT
projects MTM2016-76958-C2-1-P and PID2019-103961GB-C21 (Spain) and 
Junta de Extremadura project IB20038.
The second author was supported by the grant BES-2017-079901 related to the
project MTM2016-76958-C2-1-P.

\bibliographystyle{abbrv}

\bibliography{Fuzzy}

\begin{thebibliography}{10}

\bibitem{bag2003finite}
T.~Bag and S.~K. Samanta.
\newblock Finite dimensional fuzzy normed linear spaces.
\newblock {\em Journal of Fuzzy Mathematics}, 11(3):687--706, 2003.

\bibitem{bag2005bounded}
T.~Bag and S.~K. Samanta.
\newblock Fuzzy bounded linear operators.
\newblock {\em Fuzzy sets and Systems}, 151(3):513--547, 2005.

\bibitem{Crou}
J.-P. Crouzeix.
\newblock Continuity and differentiability of quasiconvex functions.
\newblock In {\em Handbook of generalized convexity and generalized
  monotonicity}, pages 121--149. Springer, 2006.

\bibitem{Font2017}
J.~J. Font, J.~Galindo, S.~Macario, and M.~Sanch\'is.
\newblock {M}azur-{U}lam type theorems for fuzzy normed spaces.
\newblock {\em Journal of Nonlinear Sciences and Applications},
  10(8):4499--4506, 2017.

\bibitem{Handbook}
N.~Hadjisavvas, S.~Koml{\'o}si, and S.~S. Schaible.
\newblock {\em Handbook of generalized convexity and generalized monotonicity},
  volume~76.
\newblock Springer Science \& Business Media, 2006.

\bibitem{MMM}
A.~Mirmostafaee, M.~Mirzavaziri, and M.~Moslehian.
\newblock Fuzzy stability of the {J}ensen functional equation.
\newblock {\em Fuzzy Sets and Systems}, 159(6):730--738, 2008.
\newblock Theme: Generalized Measures and Metrics.

\bibitem{nadaban}
S.~N{\u{a}}d{\u{a}}ban and I.~Dzitac.
\newblock Atomic decompositions of fuzzy normed linear spaces for wavelet
  applications.
\newblock {\em Informatica}, 25(4):643--662, 2014.

\bibitem{Ramik}
J.~Ram{\'\i}k and M.~Vlach.
\newblock Triangular norms and t-quasiconcave functions.
\newblock In {\em Generalized Concavity in Fuzzy Optimization and Decision
  Analysis}, pages 73--99. Springer, 2002.

\bibitem{MatrixNorm}
J.~Recasens.
\newblock On the relationship between positive semi-definite matrices and
  t-norms.
\newblock {\em Fuzzy Sets and Systems}, 2021.

\bibitem{HBF1999}
G.~S. Rhie and I.~A. Hwang.
\newblock On the fuzzy {H}ahn--{B}anach {T}heorem--an analytic form.
\newblock {\em Fuzzy Sets and Systems}, 108(1):117--121, 1999.

\bibitem{Deco2020}
J.~W. Rui~Gao, Xinxin~Li.
\newblock The decomposition {T}heorem for a fuzzy quasinorm.
\newblock {\em Journal of Mathematics}, 2020:\ 7 pages, 2020.

\bibitem{sadeqi}
I.~Sadeqi and F.~S. Kia.
\newblock Fuzzy normed linear space and its topological structure.
\newblock {\em Chaos, Solitons \& Fractals}, 40(5):2576--2589, 2009.

\bibitem{QFN}
Z.~Wang.
\newblock Fuzzy approximate $m$-mappings in quasi fuzzy normed spaces.
\newblock {\em Fuzzy Sets and Systems}, 406:82--92, 2021.

\end{thebibliography}

\end{document}